\documentclass[preprint,12pt]{amsart}

\usepackage{amssymb}
\usepackage{amsthm}
\usepackage{url}
\usepackage{amsmath}
\usepackage{latexsym}
\usepackage{amsbsy}
\usepackage[dvips]{color}

\usepackage{tikz}

\theoremstyle{plain}
\newtheorem{thm}{Theorem}[section]
\newtheorem{prop}[thm]{Proposition}
\newtheorem{lem}[thm]{Lemma}

\theoremstyle{definition}

\theoremstyle{remark}
\newtheorem{rem}{Remark}[section]
\newtheorem{ex}[rem]{Example}

\newcommand{\forme}[1]{}
\newcommand{\gn}[1]{\langle {#1}\rangle}

\title{The number of ideals of $\mathbb{Z}[x]$ containing $x(x-\alpha)(x-\beta)$
with given index}

\author{Mitsugu Hirasaka}
\address{Department of Mathematics, College of Sciences, Pusan National University, 2, Busandaehak-ro 63beon-gil, Geumjung-gu, Busan 46241, Korea \vspace{1.5ex}}

\email{hirasaka@pusan.ac.kr}
\thanks{This research was supported by Basic Research Program through
the National Research Foundation of Korea(NRF) funded by the Ministry of Education,
Science and Technology (grant number NRF-2013R1A1A2012532)}

\author{Semin Oh}
\email{semin@pusan.ac.kr}

\date{\today}

\begin{document}

\maketitle

\begin{abstract}
It is well-known that a connected regular graph is strongly-regular
if and only if its adjacency matrix has exactly three eigenvalues.
Let $B$ denote an integral square matrix and $\gn{B}$
denote the subring of the full matrix ring generated by $B$.
Then $\gn{B}$ is a free $\mathbb{Z}$-module of finite rank, which guarantees that
there are only finitely many ideals of $\gn{B}$ with given finite index.
Thus, the formal Dirichlet series $\zeta_{\gn{B}}(s)=\sum_{n\geq 1}a_nn^{-s}$ is well-defined
where $a_n$ is the number of ideals of $\gn{B}$ with index $n$.
In this article we aim to find an explicit form of $\zeta_{\gn{B}}(s)$ when
$B$ has exactly three eigenvalues all of which are integral, e.g., the adjacency matrix of
a strongly-regular graph which is not a conference graph with a non-squared number of vertices.
By isomorphism theorem for rings, $\gn{B}$ is isomorphic to $\mathbb{Z}[x]/m(x)\mathbb{Z}[x]$ where
$m(x)$ is the minimal polynomial of $B$ over $\mathbb{Q}$, and $\mathbb{Z}[x]/m(x)\mathbb{Z}[x]$
is isomorphic to $\mathbb{Z}[x]/m(x+\gamma)\mathbb{Z}[x]$ for each $\gamma\in \mathbb{Z}$.
Thus, the problem is reduced to counting the number of ideals of $\mathbb{Z}[x]/x(x-\alpha)(x-\beta)\mathbb{Z}[x]$
with given finite index where $0,\alpha$ and $\beta$ are distinct integers.
\end{abstract}

\section{Introduction}
For an integral square matrix $B\in M_v(\mathbb{Z})$ we denote by $\gn{B}$
the subring of $M_v(\mathbb{Z})$ generated by $B$, i.e.,
\[\gn{B}=\{f(B)\mid f(x)\in \mathbb{Z}[x]\}.\]
Since $B$ is integral, the minimal polynomial $m_B(x)$ of $B$ over $\mathbb{Q}$ is a monic integral polynomial.
By ring isomorphism theorem for the ring homomorphism from $\mathbb{Z}[x]\to \gn{B}$ defined by
$f(x)\mapsto f(B)$, we have
\[\gn{B}\simeq \mathbb{Z}[x]/m_B(x)\mathbb{Z}[x].\]
Thus, $\gn{B}$ is a free $\mathbb{Z}$-module of finite rank, and hence,
for each positive integer $n$ there are only finitely many ideals of $\gn{B}$
with index $n$. Therefore, the formal Dirichlet series $\zeta_{\gn{B}}(s)=
\sum_{n\geq 1}a_nn^{-s}$
is well-defined where $a_n$ is the number of ideals of $\gn{B}$ with index $n$.
In this article we aim to find an explicit form of $\zeta_{\gn{B}}(s)$ when $m_B(x)=x(x-\alpha)(x-\beta)$ and
$0,\alpha$ and $\beta$ are distinct integers.
In Section~2 we will show that $\zeta_{\gn{B}}(s)$ can be written as
\begin{equation}\label{eq:intro}
\prod_{p\mid \alpha\beta(\beta-\alpha)}\delta_p(p^{-s})\cdot \zeta(s)^3
\end{equation}
where $\zeta(s)$ is the Riemann zeta function, $p$ runs through the prime divisors of $\alpha\beta(\beta-\alpha)$
and $\delta_p(x)$ is a polynomial in $\mathbb{Z}[x]$. Moreover, we will find an explicit form of $\zeta_p(x)$ in Section~3.
Since
\[\mathbb{Z}[x]/m_B(x)\mathbb{Z}[x]\simeq \mathbb{Z}[x]/m_B(x+\gamma)\mathbb{Z}[x]\:\:\mbox{for each $\gamma\in \mathbb{Z}$},\]
we can apply our main result for finding $\zeta_{\gn{B}}(s)$ when $B$ is the adjacency matrix of
a simple graph with exactly three distinct eigenvalues all of which are integral (see \cite{Dam} and \cite{KM}
to find examples of non-regular graphs with three eigenvalues).
\begin{ex}\label{ex:quad}
If $B$ is the adjacency matrix of a cycle of length $4$, then
\[\zeta_{\gn{B}}(s)=(16\cdot 2^{-8s}-16\cdot 2^{-7s}+12\cdot 2^{-6s}+2\cdot 2^{-4s}+3\cdot 2^{-2s}-2\cdot 2^{-s}+1)\zeta(s)^3.\]
This formula is obtained by the method mentioned in \cite{solomon} with the aid of computer (see \cite{oh}),
and obtained theoretically as a corollary of Proposition~\ref{prop:003}.
\end{ex}
\begin{ex}\label{ex:peter}
If $B$ is the adjacency matrix of the peterson graph, then
\[\zeta_{\gn{B}}(s)=\prod_{p=2,3,5}(1-p^{-s}+p^{1-2s})\cdot \zeta(s)^3.\]
This formula is obtained theoretically through the method given in the paper (see Proposition~\ref{prop:001}).
\end{ex}
\begin{rem}
It is well-known that
a simple connected graph has exactly two eigenvalues if and only if it is a complete graph with more than one vertices.
In \cite[Thm.~1.2]{hh}, the explicit form of $\zeta_{\gn{B}}(s)$ is found when $B$ is the adjacency matrix of a complete graph.
\end{rem}

It must be noted that \cite{solomon} is the first paper to establish
a way to find zeta functions of lattices over $\mathbb{Z}$-orders,
and \cite[Thm.~4]{solomon}
shows that the zeta function of the group ring $\mathbb{Z}G$, where $G$ is a group of prime order $p$,
equals
\begin{equation}\label{eq:int2}
(1-p^{-s}+p^{1-2s})\zeta(s)\zeta_F(s)
\end{equation}
where $\zeta_F(s)$ is the
Dedekind zeta function of $F=\mathbb{Q}(\xi_p)$, and $\xi_p$ is a primitive $p$-th root of unity
(see \cite{reiner}, \cite {take} and \cite{hironaka} for other groups).
Group rings are generalized as Schur rings, Hecke rings and adjacency algebras of
association schemes (see \cite{bi}, \cite{Zie1} and \cite{Zie2} to study basic concepts for association schemes). Recently, in \cite{hh}, as a generalization of (\ref{eq:int2}),
the zeta function of the adjacency algebra $\mathbb{Z}S$ of an association scheme $(X,S)$ with $|X|=p$ is computed as follows:
\begin{equation}\label{eq:int3}
(1-p^{-s}+p^{1-2s})\zeta(s)\zeta_K(s)
\end{equation}
where $K$ is the minimal splitting field of $\mathbb{Z}S$.

The topic given in this article is motivated from \cite{solomon} and
continued from \cite{hh} and \cite{oh}. Moreover, the procedure to obtain (\ref{eq:intro})
is based on \cite{solomon}, and the authors got an idea
to compute $\delta_p(x)$ from \cite{reiner} and \cite{take}.

Remark that it is not so easy to find an explicit form of Dedekind zeta functions,
so that (\ref{eq:int3}) does not help to count the number of ideals with given index unless
$\zeta_K(s)$ is specified.

In Section~2, we prepare basic results to reach our main result, and in Section~3
we show that the zeta function of the ring given in the title can be written as the sum of
seven equations, each of which is also obtained as rational function of $p^{-s}$,
and in Section~4 we show that explicit forms of $\zeta_{\gn{B}}(s)$ for some matrices $B$
with exactly three eigenvalues all of which are integral.

\section{Preliminaries}
We use the same notation as in \cite{solomon} to start from the following theorem:
\begin{thm}[\cite{solomon}]\label{thm:solomon}
Let $A$ be a semisimple algebra over $\mathbb{Q}$, and
let $\Lambda$ be a $\mathbb{Z}$-order of $A$.
Let $V$ be an $A$-module and let $L$ be a $\Lambda$-lattice on $V$. There exists for each prime
$p\in B$ a rational function $\delta_p(x)$ in an indeterminate $x$ such that
\[\delta_p(p^{-s})=\frac{\zeta_{L_p}(s;\Lambda_p)}{\zeta_V(s)_p}\:\:\mbox{and }
\zeta_L(s)=\prod_{p\in B}\delta_p(p^{-s})\zeta_V(s).\]
\end{thm}
In the above theorem $B$ is the set of primes uniquely determined by $\Lambda$ (see \cite{solomon} for
the definitions of $B$, $\zeta_L(s)$, $\zeta_{L_p}(s;\Lambda_p)$, $\zeta_V(s)_p$ and $\zeta_V(s)$).

For the remainder of this section we assume that
\[A=\mathbb{Q}[x]/x(x-\alpha)(x-\beta)\mathbb{Q}[x]\]
where $0,\alpha$ and $\beta$ are distinct integers,
$V=A$ is viewed as a regular $A$-module,
\[\Lambda=\mathbb{Z}[x]/x(x-\alpha)(x-\beta)\mathbb{Z}[x]\] and
$L=\Lambda$ is viewed as a regular left $\Lambda$-module.
We note that $A$ is a finite-dimensional semisimple $\mathbb{Q}$-algebra,
$\Lambda$ is identified with a $\mathbb{Z}$-order of $A$
and $L$ is a $\Lambda$-lattice
on $V$ so that Theorem~\ref{thm:solomon} can be applied for $(A,\Lambda, V,L)$.
\begin{lem}\label{lem:B}
Under the same assumption as in Theorem~\ref{thm:solomon} we have the following:
\begin{enumerate}
\item $B$ is the set of
prime divisors of $\alpha \beta(\beta-\alpha)$;
\item  $\zeta_V(s)=\zeta(s)\cdot\zeta(s)\cdot \zeta(s)$ where
$\zeta(s)$ is the Riemann zeta function;
\item $\zeta_V(s)_p=(1-p^{-s})^{-3}$.
\end{enumerate}
\end{lem}
\begin{proof}
(i) Since
\[\left\{\frac{\beta x-x^2}{\alpha(\beta-\alpha)}, \frac{\alpha x-x^2}{\beta(\alpha-\beta)},
\frac{(x-\alpha)(x-\beta)}{\alpha \beta}\right\}\]
is a unique basis for $A$ consisting of primitive idempotents,
it spans a maximal order $\Gamma$ of $A$ containing $\Lambda$.
This implies that
$|\Gamma:\Lambda|$
divides $|\alpha\beta(\beta-\alpha)|$. According to \cite[Rem.]{solomon} $B$ coincides with
the set of prime divisors
of the index of $\Lambda$
in a maximal $\mathbb{Z}$-order $A$ containing $\Lambda$.
This completes the proof of (i).

(ii) Since $A\simeq \mathbb{Q}\oplus  \mathbb{Q}\oplus  \mathbb{Q} $ as
$\mathbb{Q}$-algebras and $V$ is isomorphic to ${}_AA$,
the lemma follows from the definition of $\zeta_V(s)$.

(iii) follows from (ii) and the definition of $\zeta_V(s)_p$.
\end{proof}
According to \cite{solomon} we set
$\Lambda_p=\mathbb{Z}_p \Lambda$ and $L_p=\mathbb{Z}_pL$
where
$\mathbb{Z}_p$ is the localization of $\mathbb{Z}$ at a prime $p$, that is,
\[\mathbb{Z}_p=\{a/b\mid a,b\in \mathbb{Z};p\nmid b\}.\]
It is well-known that each nonzero ideal of $\mathbb{Z}_p$ forms
$p^r\mathbb{Z}_p$ for a unique non-negative integer $r$.
For nonzero $a\in \mathbb{Z}_p$ we denote by
$[a]$ the unique positive integer $r$ such that $a\mathbb{Z}_p=p^r\mathbb{Z}_p$,
and we define $[0]$ to be $+\infty$ for convenience.
The following is a well-known result on valuation rings:
\begin{lem}\label{lem:part}
For all $a,b\in \mathbb{Z}_p$
we have the following:
\begin{enumerate}
\item $[ab]=[a][b]$;
\item $[a+ b]=\min\{[a],[b]\}$ if $[a]\ne [b]$.
\end{enumerate}
\end{lem}
\forme{
\begin{proof}
These are well-known properties on valuation rings
(i) If one of $a$ and $b$ is zero, then both side are equal to $+\infty$.
Suppose both $a$ and $b$ are nonzero, $[a]=r$ and $[b]=s$.
Then $ab\in p^{r+s}\mathbb{Z}_p$, so $[ab]\geq r+s$.
Since $\mathbb{Z}_p$ is a principal ideal domain, it is a unique factorization domain.
This implies that $ab\notin p^{r+s+1}\mathbb{Z}_p$, and hence, $[ab]=r+s$.

(ii) If exactly one of $a$ and $b$ is zero, then the equality holds.
Suppose both $a$ and $b$ are nonzero, $[a]=r$ and $[b]=s$ with $r<s$.
Then $a+b=p^r(c+dp^{s-r})$ for some $c\in \mathbb{Z}_p\setminus p\mathbb{Z}_p$ and
$d\in \mathbb{Z}_p$. Since $c+dp^{s-r}$ is an invertible element, $[a+b]=r$.
\end{proof}
}

For the remainder of this section we shall write $L_p$ as $R$ for convenience, i.e.,
\[R=\mathbb{Z}_p[x]/x(x-\alpha)(x-\beta)\mathbb{Z}_p[x].\]
Setting $f(x)=x(x-\alpha)(x-\beta)$, since
\[\mathbb{Z}[x]/f(x)\mathbb{Z}[x]\simeq \mathbb{Z}[x]/f(x+\alpha)\mathbb{Z}[x]
\simeq \mathbb{Z}[x]/f(x+\beta)\mathbb{Z}[x],\]
we may assume that
$[\alpha]\leq [\beta]\leq [\beta-\alpha]$ without loss of generality.
If $[\alpha]< [\beta]$, then $[\beta-\alpha]=[\alpha]$ by Lemma~\ref{lem:part}(ii),
which contradicts $[\alpha]< [\beta]\leq [\beta-\alpha]$.
Thus, we have
\begin{equation}\label{eq:abc}
[\alpha]=[\beta]\leq [\beta-\alpha].
\end{equation}

Let $I$ be a full ideal of $R$ and fix a $\mathbb{Z}_p$-basis $\{1,x,x(x-\alpha)\}$ for $R$.
Then we can apply a well-known
result on modules over principal ideal domains for $I$ being a free $\mathbb{Z}_p$-module of rank three
to find $(r_1,r_2,r_3)\in \mathbb{N}$ and
$(a_1,a_2,a_3)\in \mathbb{Z}_p^3$ such that
\[\{p^{r_1}+a_1x+a_2x(x-\alpha), p^{r_2}x+a_3x(x-\alpha),p^{r_3}x(x-\alpha)\}\]
is a $\mathbb{Z}_p$-basis for $I$, and
$(r_1,r_2,r_3)\in \mathbb{N}$ is uniquely determined by $I$
as $p^{r_1}\mathbb{Z}_p=\{g(0)\mid g\in I\}$, $p^{r_2}\mathbb{Z}_p=\{g(\alpha)\mid g\in R; xg\in I\}$
and
\[p^{r_3}\mathbb{Z}_p=\{g(\beta)\mid g\in R;a(x-\alpha)g\in I\}.\]
We say that an ideal $I$ of $R$ is of \textit{type} of $(r_1,r_2,r_3)$ in this case.
On the other hand, it is easily proved that
\[\{p^{r_1}+b_1x+b_2x(x-\alpha), p^{r_2}x+b_3x(x-\alpha),p^{r_3}x(x-\alpha)\}\]
is also a basis for $I$ if and only if
\[a_3-b_3\in p^{r_3}\mathbb{Z}_p\:\:\mbox{and}\:\:
(a_1,a_2)-(b_1,b_2)\in  \mathbb{Z}_p(p^{r_2},a_3)+\mathbb{Z}_p(0,p^{r_3}).\]
We shall count the number of ideals of type $(r_1,r_2,r_3)$,
which equals the number of equivalence classes defined above on the set of
$(a_1,a_2,a_3)\in \mathbb{Z}_p^3$ such that
\begin{equation}\label{eq1}
p^{r_2}+a_3(\beta-\alpha)\in p^{r_3}\mathbb{Z}_p;
\end{equation}

\begin{equation}\label{eq2}
p^{r_1}+\alpha a_1\in p^{r_2}\mathbb{Z}_p;
\end{equation}

\begin{equation}\label{eq3}
(p^{r_2}-\alpha a_3)a_1+\beta p^{r_2}a_2-p^{r_1}a_3\in p^{r_2+r_3}\mathbb{Z}_p
\end{equation}
since these equations hold if and only if
\[x\{p^{r_1}+a_1x+a_2x(x-\alpha), p^{r_2}x+a_3x(x-\alpha),p^{r_3}x(x-\alpha)\}\subseteq I.\]

For $(r_1,r_2,r_3)\in \mathbb{N}^3$ we define $S(r_2,r_3)$ to be
\begin{equation*}%\label{eq:sol}
S(r_2,r_3)=\{a_3\in \mathbb{Z}_p\mid \mbox{(\ref{eq1}) holds}\}.
\end{equation*}
and
for $a_3\in S(r_2,r_3)$
we define
$S(r_1,r_2,r_3,a_3)$ to be
\begin{equation*}%\label{eq:sol2}
S(r_1,r_2,r_3,a_3)=\{(a_1,a_2)\in \mathbb{Z}_p^2\mid
\mbox{(\ref{eq2}) and (\ref{eq3}) hold}\}
\end{equation*}

Since
 \[p^{r_2}+(w+b)(\beta-\alpha)=p^{r_2}+w(\beta-\alpha)+b(\beta-\alpha)\in
p^{r_3}\mathbb{Z}_p\]
for all $w\in S(r_2,r_3)$
and $b\in p^{r_3}\mathbb{Z}_p$,
$S(r_2,r_3)$ is a union of cosets of $p^{r_3}\mathbb{Z}_p$ (possibly empty).
We define $\bar{S}(r_2,r_3)$ to be
the set of cosets of $p^{r_3}\mathbb{Z}_p$ contained in $S(r_2,r_3)$.

\begin{lem}\label{lem:union}
For $a_3,b_3\in S(r_2,r_3)$ with $a_3-b_3\in p^{r_3}\mathbb{Z}_p$ we have
\[S(r_1,r_2,r_3,a_3)=S(r_1,r_2,r_3,b_3),\]
which is a union of cosets of
$\mathbb{Z}_p(p^{r_2},a_3)+\mathbb{Z}_p(0,p^{r_3})$ (possibly empty).
\end{lem}
\begin{proof}
Replacing $a_3$ by $b_3+p^{r_3}t$ in (\ref{eq3}) with $(a_1,a_2)\in S(r_1,r_2,r_3,a_3)$
we obtain that
\[(p^{r_2}-\alpha b_3)a_1+\beta p^{r_2}a_2-p^{r_1}b_3-p^{r_3}t(\alpha a_1+p^{r_1})\in p^{r_2+r_3}\mathbb{Z}_p.\]
Since $\alpha a_1+p^{r_1}\in p^{r_2}\mathbb{Z}_p$ by (\ref{eq2}), it follows that $(a_1,a_2)\in S(r_1,r_2,r_3,b_3)$.
This implies that $S(r_1,r_2,r_3,a_3)\subseteq S(r_1,r_2,r_3,b_3)$, and similarly we have $S(r_1,r_2,r_3,b_3)\subseteq S(r_1,r_2,r_3,a_3)$.
This proves the first statement.

In order to the second statement it suffices to show that
\[(y,z)+b(p^{r_2},a_3)+c(0,p^{r_3})=(y+bp^{r_2},z+ba_3+cp^{r_3})\in S(a,r_1,r_2,r_3)\]
whenever $(y,z)\in S(a,r_1,r_2,r_3)$
and $(b,c)\in \mathbb{Z}_p\times \mathbb{Z}_p$.
In fact,
\[p^{r_1}+\alpha(y+bp^{r_2})
=p^{r_1}+\alpha y+\alpha  bp^{r_2}\in
p^{r_2}\mathbb{Z}_p, \mbox{~~~and}\]
\[
(p^{r_2}-\alpha a_3)bp^{r_2}+\beta p^{r_2}(ba_3+cp^{r_3})
=bp^{r_2}(p^{r_2}+(\beta-\alpha)a)+c\beta p^{r_2+r_3}
\in p^{r_2+r_3}\mathbb{Z}_p\]
since $p^{r_2}+(\beta-\alpha)a\in p^{r_3}\mathbb{Z}_p$ by (\ref{eq1}).
\end{proof}
For $\bar{a}\in \bar{S}(r_2,r_3)$ where $\bar{a}=a+p^{r_3}\mathbb{Z}_p$
we denote the set of cosets of $\mathbb{Z}_p(p^{r_2},a)+\mathbb{Z}_p(0,p^{r_3})$ contained in
$S(r_1,r_2,r_3,a)$
by $\bar{S}(r_1,r_2,r_3,a)$, which is well-defined by Lemma~\ref{lem:union}.
Therefore, the number of ideals of type $(r_1,r_2,r_3)$ is given by the following formula:
\begin{equation}\label{eq:type}
\sum_{\bar{a}\in \bar{S}(r_2,r_3)}|\bar{S}(r_1,r_2,r_3,a)|.
\end{equation}
We shall find an explicit form of of $\zeta_R(s)$ as follows:
\begin{align*}
\zeta_R(s) & =\sum_{I}|R:I|^{-s}  \\
           & \mbox{where $I$ runs through the full ideals of $R$} \\
    &=\sum_{(r_1,r_2,r_3)\in \mathbb{N}^3}\sum_{\bar{a}\in \bar{S}(r_2,r_3)}|\bar{S}(r_1,r_2,r_3,a)|p^{-s(r_1+r_2+r_3)} \\
    & \mbox{by (\ref{eq:type}) and since the index of an ideal of type $(r_1,r_2,r_3)$ equals $p^{r_1+r_2+r_3}$}\\
    & = \sum_{(r_3,r_2)\in \mathbb{N}^2}\sum_{\bar{a}\in \bar{S}(r_2,r_3)}\sum_{r_1\in C(r_2,r_3,a)}|\bar{S}(r_1,r_2,r_3,a)|p^{-s(r_1+r_2+r_3)}\\
    & \mbox{where $C(r_2,r_3,a):=\{r_1\in \mathbb{N}\mid \bar{S}(r_1,r_2,r_3,a)\ne \emptyset\}$ }\\
    & = \sum_{(r_3,r_2)\in \mathbb{N}^2}E(r_2,r_3)p^{-s(r_1+r_2+r_3)}\\
    & \mbox{where $E(r_2,r_3):=\sum_{\bar{a}\in \bar{S}(r_2,r_3)}
\sum_{r_1\in C(r_2,r_3,a)}|\bar{S}(r_1,r_2,r_3,a)|$}\\
& \mbox{ is viewed as an operator.}
    \end{align*}
\begin{lem} \label{lem:S}
We have
 \[S(r_2,r_3)=
\begin{cases}
  \emptyset & \textup{ if } r_2 < \textup{min}\{[\beta-\alpha],r_3\} \\
  \mathbb{Z}_p & \textup{ if } r_3 \leq \textup{min}\{[\beta - \alpha],r_2\}\\
  p^{r_3-[\beta-\alpha]}\mathbb{Z}_p & \textup{ if } [\beta-\alpha] \leq r_3 \leq r_2 \\
  p^{r_3-[\beta-\alpha]}\mathbb{Z}_p - \frac{p^{r_2}}{\beta - \alpha} & \textup{ if } [\beta-\alpha] \leq r_2 < r_3.
\end{cases}\]
\end{lem}
\begin{proof}
Let $a\in  S(r_2,r_3)$.
Then, by (\ref{eq1}),
\[p^{r_2}\in (\beta-\alpha)a+p^{r_3}\mathbb{Z}_p\subseteq
 -(\beta-\alpha)\mathbb{Z}_p+p^{r_3}\mathbb{Z}_p=
p^{\min\{r_3, [\beta-\alpha]\}}\mathbb{Z}_p.\]
Thus, $r_2\geq \min\{[\beta-\alpha],r_3\}$ if and only if $S(r_2,r_3)\ne \emptyset$.

Suppose $r_2\geq \min\{[\beta-\alpha],r_3\}$.
If $r_3\leq [\beta-\alpha]$, then each element $w\in \mathbb{Z}_p$
satisfies $p^{r_2}+(\beta-\alpha)w\in p^{r_3}\mathbb{Z}_p$,
implying $S(r_2,r_3)=\mathbb{Z}_p$.
If $[\beta-\alpha]<r_3$, then $p^{r_2}/(\beta-\alpha)$, $p^{r_3}/(\beta-\alpha)\in \mathbb{Z}_p$, and hence,
\[S(r_2,r_3)=p^{r_3-[\beta-\alpha]}\mathbb{Z}_p-p^{r_2}/(\beta-\alpha).\]
Moreover, it equals $p^{r_3-[\beta-\alpha]}\mathbb{Z}_p$ if $r_3\leq r_2$.
This completes the proof.
\end{proof}

Applying Lemma~\ref{lem:S} for $(r_3,r_2)\in \mathbb{N}^2$ with $r_2<\min\{[\beta-\alpha],r_3\}$
we conclude that $S(r_2,r_3)=\emptyset$, and hence, $E(r_2,r_3)=0$.
For other cases we can take a finite set of representatives of $S(r_2,r_3)$ modulo $p^{r_3}\mathbb{Z}_p$
as follows:
\begin{enumerate}
\item $\{1,2,\ldots, p^{r_3}\}$ if $r_3 \leq \textup{min}\{[\beta - \alpha],r_2\}$;
\item $\{p^{r_3-[\beta-\alpha]}i\mid i=1,2,\ldots, p^{[\beta-\alpha]}\}$ if $ [\beta-\alpha] \leq r_3 \leq r_2 $;
\item $\{p^{r_3-[\beta-\alpha]}i - \frac{p^{r_2}}{\beta-\alpha}\mid i=1,2,\ldots, p^{[\beta-\alpha]}\}$ if $[\beta-\alpha] \leq r_2< r_3$.
\end{enumerate}
For the remainder of this section we prepare some lemmas to compute $E(r_2,r_3)$:
\begin{lem}\label{lem:cr1}
For $a,b\in S(r_2,r_3)$ with $a-b\in p^{r_3}\mathbb{Z}_p$ we have
\[C(r_2,r_3,a)=C(r_2,r_3,b)=\{r\in \mathbb{N}
\mid r\geq\min\{[p^{r_2}-\alpha a],2[\alpha],r_3+[\alpha]\}\}.\]
\end{lem}
\begin{proof}
Suppose $(y,z)\in S(r_1,r_2,r_3,a)$.
Then
\[p^{r_1}=-\alpha y+p^{r_2}\mathbb{Z}_p\subseteq
p^{[\alpha]}\mathbb{Z}_p+p^{r_2}\mathbb{Z}_p
=p^{\min\{[\alpha],r_2\}}\mathbb{Z}_p.\]
This implies that $r_1\geq \min\{[\alpha],r_2\}$.
Since
\begin{equation}\label{eq:x1}
\mbox{$p^{r_1}+\alpha y=p^{r_2}u$ for some $u\in \mathbb{Z}_p$ by (\ref{eq2})}
\end{equation}
and \[(p^{r_2}-\alpha a)\alpha y+\alpha \beta p^{r_2}z-\alpha p^{r_1} a\in p^{[\alpha]+r_2+r_3}\mathbb{Z}_p\:\:\mbox{by (\ref{eq3})},\]
we have
\[(p^{r_2}-\alpha a)(p^{r_2} u-p^{r_1})+\alpha \beta p^{r_2} z-p^{r_1} \alpha a
\in p^{[\alpha]+r_2+r_3}\mathbb{Z}_p,\]
and hence,
\begin{equation}\label{eq:x2}
(p^{r_2}-\alpha a )p^{r_2} u+\alpha \beta p^{r_2} z-p^{r_1+r_2}\in
p^{[\alpha]+r_2+r_3}\mathbb{Z}_p.
\end{equation}
Thus,
\[p^{r_1+r_2}\in (p^{r_2}-\alpha a)p^{r_2}\mathbb{Z}_p+\alpha \beta p^{r_2}\mathbb{Z}_p+
p^{[\alpha]+r_2+r_3}\mathbb{Z}_p,\]
which implies
$r_1\geq \min\{[p^{r_2}-\alpha a],[\alpha]+[\beta],[\alpha]+r_3\} $.
Conversely, the two inequalities imply
the existence of $(y,u,z)$ which satisfies (\ref{eq:x1}) and (\ref{eq:x2}).
So, it suffices to show that
\[\min\{r_2,[\alpha]\}\leq\min\{[p^{r_2}-\alpha a],2[\alpha],r_3+[\alpha]\}.\]
If $r_2<[\alpha]$,
then $[p^{r_2}-\alpha a]=r_2$ by Lemma~\ref{lem:part}(ii), and hence we are done.
If $r_2\geq [\alpha]$,
then $[p^{r_2}-\alpha a]\geq [\alpha]$ and $[\alpha]\leq \min\{2[\alpha],r_3+[\alpha]\}$ as desired.

Since $a-b\in p^{r_3}\mathbb{Z}_p$,
$a=b+p^{r_3}t$ for some $t\in \mathbb{Z}_p$.
It is easy to show that $[p^{r_2}-\alpha a]=[p^{r_2}-\alpha b]$ if $[p^{r_2}-\alpha a]<[\alpha]+r_3$,
and
\[\min\{[p^{r_2}-\alpha b],2[\alpha],r_3+[\alpha]\}=\min\{2[\alpha],[\alpha]+r_3\}\]
if $[p^{r_2}-\alpha a]\geq [\alpha]+r_3$.
This completes the proof.
\end{proof}

\begin{lem}\label{lem:nr4}
For all $(r_1,r_2,r_3)\in \mathbb{N}^3$ and $a\in S(r_2,r_3)$
with $r_1\in C(r_2,r_3,a)$ we have
\[
|\bar{S}(r_1,r_2,r_3,a)|=
p^{\min\{[p^{r_2}-\alpha a],2[\alpha],[\alpha]+r_3\}}.\]
\end{lem}
\begin{proof}
Since $\bar{S}(r_1,r_2,r_3,a)\ne \emptyset$, we can take $(y_0,z_0)\in S(r_1,r_2,r_3,a)$ so that
$\bar{S}(r_1,r_2,r_3,a)-(y_0,z_0)$ is an additive subgroup of $\mathbb{Z}_p\times \mathbb{Z}_p$.
For short we set
\[\mbox{$H=\bar{S}(r_1,r_2,r_3,a)-(y_0,z_0)$,
$K=\mathbb{Z}_p(p^{r_2},a)+\mathbb{Z}_p(0,p^{r_3})$ and $G=\mathbb{Z}_p\times \mathbb{Z}_p$.}\]
By Lemma~\ref{lem:union},
we have
\begin{equation}\label{eq:q0}
|\bar{S}(r_1,r_2,r_3,a)|=|H:K|=|G:K|/|G:H|
=p^{r_2+r_3}/|G:H|.
\end{equation}
Note that
\[H=\{(y,z)\in \mathbb{Z}_p\times \mathbb{Z}_p\mid dy\in e\mathbb{Z}_p, gy+bz\in c\mathbb{Z}_p\}\]
where $g=p^{r_2}-\alpha a$, $b=\beta p^{r_2}$, $c=p^{r_2+r_3}$, $d=\alpha$ and $e=p^{r_2}$.

We claim that
\[\{y\in \mathbb{Z}_p\mid dy\in e\mathbb{Z}_p\}=p^{\max\{0,[e]-[d]\}}\mathbb{Z}_p.\]
It is clear if $[e]<[d]$.
Since $\mathbb{Z}_p$ is an integral domain $e\ne 0$, $dy\in e\mathbb{Z}_p$ is equivalent to
$y\in e/d\mathbb{Z}_p$ if $[e]\geq [d]$. Since $e/d\mathbb{Z}_p=p^{[e]-[d]}\mathbb{Z}_p$,
the claim holds.

We claim that $|G:H|$ equals
\begin{equation*}
\begin{cases}
 \mathrm{(i)}\:\: p^{[c]-[g]} &
\mbox{ ~~if~~ } [g]+\max\{0,[e]-[d]\}\leq [b]\leq [c]\\

\mathrm{(ii)}\:\:p^{\max\{0,[e]-[d]\}+[c]-[b]} &
\mbox{ ~~if~~ }  [b]\leq [a]+\max\{0,[e]-[d]\}\leq [c]\\

\mathrm{(iii)}\:\:p^{\max\{0,[c]-[b]\}+\max\{0,[e]-[d]\}} &
\mbox{ ~~if~~ } [b]\leq [c]\leq [g]+\max\{0,[e]-[d]\} \\

\mathrm{(iv)}\:\:p^{\max\{0,[c]-[g],[e]-[d]\}} &
 \mbox{ ~~if~~ } [g]+\max\{0,[e]-[d]\}\leq [c]\leq [b]\\

\mathrm{(v)}\:\:p^{\max\{0,[e]-[d]\}} &
 \mbox{ ~~if~~ } [c]\leq [g]+\max\{0,[e]-[d]\}\leq [b]\\

\mathrm{(vi)}\:\:p^{\max\{0,[e]-[d]\}} &
 \mbox{ ~~if~~ } [c]\leq [b]\leq [g]+\max\{0,[e]-[d]\}
\end{cases}
\end{equation*}
(i) By the assumption of (i), we have $[g]\leq [b]\leq [c]$, and hence,
if $(y,z)\in H$, then, for some $w\in \mathbb{Z}_p$,
\[(y,z)=((-b/a)z+(c/a) w,z)=z(-b/a,1)+w(c/a,0).\]
By the assumption, the above equation implies $y\in  p^{\max\{0,[e]-[d]\}}\mathbb{Z}_p$.
Thus, $\{(-b/g,1),(c/g,0)\}$ is a $\mathbb{Z}_p$-basis for $H$, and
$|G:H|=p^{[c]-[g]}$.

(ii) By the claim, $y=fw$ for some $w\in \mathbb{Z}_p$ where $f=p^{\max\{0,[e]-[d]\}}$.
Let $(y,z)\in H$.
Then, by the assumption of (ii),
$z=(-gf/b)w+c/bt$ for some $t\in \mathbb{Z}_p$, and
\[(y,z)=(fw,(-gf/b)w+c/bt)=w(f,-gf/b)+t(0,c/b).\]
Thus, $\{(f,-gf/b),(0,c/b)\}$ is a $\mathbb{Z}_p$-basis for $H$, and
$|G:H|=p^{[f]+[c]-[b]}$.

(iii) Note that
$H=\{(y,z)\mid dy\in e\mathbb{Z}_p, bz\in c\mathbb{Z}_p\}$.
By the claim, $H=f\mathbb{Z}_p\times g\mathbb{Z}_p$ where
$g=p^{\max\{0,[c]-[b]\}}$.
Thus, $|G:H|=p^{\max\{0,[e]-[d]\}+\max\{0,[c]-[b]\}}$.

(iv) Note that
$H=\{(y,z)\mid dy\in e\mathbb{Z}_p, ay\in c\mathbb{Z}_p\}$.
By the claim, letting $h=p^{\max\{0,[c]-[g]\}}$ we have
\[H=(f\mathbb{Z}_p\cap h\mathbb{Z}_p)\times \mathbb{Z}_p=p^{\max\{[f],[h]\}}\mathbb{Z}_p\times \mathbb{Z}_p. \]
Thus, $|G:H|=p^{\max\{0,[e]-[d],[c]-[g]\}}$.

(v),(vi)  Note that
$H=\{(y,z)\mid dy\in e\mathbb{Z}_p, gy\in c\mathbb{Z}_p\}$.
Since $[c]\leq [g]+[f]$, we have $\max\{0,[e]-[d],[c]-[g]\}=\max\{0,[e]-[d]\}$.
By the similar argument as (iv) we obtain
$|G:H|=p^{\max\{0,[e]-[d]\}}$.

Therefore, the equation given in the lemma holds as a direct consequence of (\ref{eq:q0}) and the second claim.
\end{proof}

For $(r_3,r_2)\in \mathbb{N}^2$ with $r_3\leq r_2$ and $i\in \mathbb{N}$ we set
\begin{equation*}
S_i(r_2,r_3)=\{a\in S(r_2,r_3)\mid [a]=i\}\:\:\mbox{and }\:\:\bar{S}_i(r_2,r_3)=\{\bar{a}\in S(r_2,r_3)\mid [a]=i\}
\end{equation*}
so that $\bar{S}(r_2,r_3)$ is a disjoint union of $\bar{S}_i(r_2,r_3)$ with $i=m,m+1,\ldots,r_3$ by Lemma~\ref{lem:S}
and the statement preddceding the proof of Lemma~\ref{lem:S}
where $m:=\max\{0,r_3-[\beta-\alpha]\}$.
\begin{lem}\label{lem:size}
For $(r_3,r_2)\in \mathbb{N}^2$ and $i\in \mathbb{N}$ with $r_3\leq r_2$ and $m\leq i\leq r_3$
we have
\[|\bar{S}_i(r_2,r_3)|=p^{r_3-i}-\lfloor p^{r_3-i-1}\rfloor\]
where $m:=\max\{0,r_3-[\beta-\alpha]\}$ and $\lfloor b\rfloor$ is the greatest integer $n$ with $n\leq b$ for $b\in \mathbb{R}$.
\end{lem}
\begin{proof}
If $m\leq i< r_3$,
then $|\bar{S}_i(r_2,r_3)|$ equals the number of cosets of $p^{r_3}\mathbb{Z}_p$ contained in
$p^i\mathbb{Z}_p\setminus p^{i+1}\mathbb{Z}_p$, and hence, equal to $p^{r_3-i}-p^{r_3-i-1}$.
If $i=r_3$, then $\bar{S}_i(r_2,r_3)=\{\bar{0}\}$ and hence, $|\bar{S}_i(r_2,r_3)|=1$
\end{proof}

\begin{lem}\label{lem:basic}
For $(r_3,r_2)\in \mathbb{N}^2$
and $a\in S(r_2,r_3)$ we have
\[[p^{r_2}-\alpha a]=\begin{cases}
[\alpha]+ [a] & \mbox{if $[a]<r_2-[\alpha]$}\\
r_2+j & \mbox{if $a=p^{r_2}/\alpha(1+p^jd)$ for some $d\in \mathbb{Z}_p\setminus p\mathbb{Z}_p$}\\
r_2 & \mbox{if $[a]>r_2-[\alpha]$.}
\end{cases}
\]
\end{lem}
\begin{proof}
This is an immediate consequence of Lemma~\ref{lem:part}.
\end{proof}

\begin{lem}\label{lem:basic2}
For $(r_3,r_2)\in \mathbb{N}^2$ with $[\beta-\alpha]\leq r_2<r_3$
and $a\in S(r_2,r_3)$ we have
\[[p^{r_2}-\alpha a]=r_2+[\alpha]-[\beta-\alpha].\]
\end{lem}
\begin{proof}
By Lemma~\ref{lem:S},
$a=p^{r_3-[\beta-\alpha]}t-p^{r_2}/(\beta-\alpha)$ for some $t\in \mathbb{Z}_p$, and hence,
\begin{align*}
p^{r_2}-\alpha a &=p^{r_2}-\alpha p^{r_3-[\beta-\alpha]}t +\alpha p^{r_2}/(\beta-\alpha)\\
&=p^{r_2}\beta/(\beta-\alpha)-\alpha p^{r_3-[\beta-\alpha]}t
\end{align*}
Since $r_2<r_3$ and $[\alpha]=[\beta]$ by (\ref{eq:abc}),
the lemma follows from Lemma~\ref{lem:part}.
\end{proof}

As mentioned in Lemma~\ref{lem:basic},
it is easy to compute $[p^{r_2}-\alpha a]$ if $r_2\ne [\alpha]+[a]$ by Lemma~\ref{lem:part}.
However, we need to know which $a\in S(r_2,r_3)$ attains $[p^{r_2}-\alpha a]=r_2+i$ for given $i$
and how many $\bar{a}\in \bar{S}(r_2,r_3)$ so does.
The following lemma is useful in determining them:

\begin{lem} \label{lem:X}
For $(r_3,r_2)\in \mathbb{N}^2$ with
\[0\leq r_3<[\beta-\alpha]+[\alpha], \max\{[\alpha],r_3\}\leq  r_2<[\alpha]+r_3,\]
$\bar{S}_{r_2[\alpha]}(r_2,r_3)$ is a disjoint union of $X_j$ with $0\leq j\leq r_3-r_2+[\alpha]$
where
\[X_j:=\{\bar{a} \in \bar{S}_{r_2-[\alpha]}(r_2,r_3) \mid [p^{r_2}-\alpha a] = r_2 + j\}.\]
Moreover, $|X_0|=p^{r_3-r_2+[\alpha]}-2p^{r_3-r_2+[\alpha]-1}$ and
\[|X_j|= p^{r_3-r_2+[\alpha]-j} - \lfloor p^{r_3-r_2+[\alpha]-j-1} \rfloor.\]
\end{lem}
\begin{proof}
It is clear that $X_j$ are disjoint subsets of $S_{r_2-[\alpha]}$.
Thus, it suffices to prove the statements on $|X_j|$ since
$|\bar{S}_{r_2-[\alpha]}(r_2,r_3)|=p^{r_3-r_2-[\alpha]}-p^{r_3-r_2-[\alpha]-1}$ being the summation of
the numbers given above by Lemma~\ref{lem:size}.

For $\bar{a} \in \bar{S}_{r_2-[\alpha]}(r_2,r_3)$ with $[p^{r_2}-\alpha a]=r_2+j$
there exists $b\in \mathbb{Z}\setminus p\mathbb{Z}_p$ such that $a=(p^{r_2}/\alpha)(1+p^jb)$
and $1+p^jb$ is invertible.
Conversely, if $b\in \mathbb{Z}\setminus p\mathbb{Z}_p$ and $1+p^jb\in \mathbb{Z}_p\setminus p\mathbb{Z}_p$,
then $\bar{a}\in X_j$ where $a=(p^{r_2}/\alpha)(1+p^jb)$.
Note that $1+p^jb$ is invertible unless $j=0$ and $b\in-1+p\mathbb{Z}_p$.
Thus,
\begin{align*}
|X_0|&=\frac{|\{\bar{b}\in \mathbb{Z}_p/p^{r_3}\mathbb{Z}_p\mid b\notin p\mathbb{Z}_p\cup(-1+p\mathbb{Z}_p)\}|}
{|\{\bar{a}\in \mathbb{Z}_p/p^{r_3}\mathbb{Z}_p\mid p^{r_2-[\alpha]}a\in p^{r_3}\mathbb{Z}_p\}|}\\
&=\frac{p^{r_3}-2p^{r_3-1}}{p^{r_2-[\alpha]}}=p^{r_3-r_2+[\alpha]}-2p^{r_3-r_2+[\alpha]-1},
\end{align*}
and similarly we have, for $j$ with $1\leq j\leq r_3-r_2+[\alpha]$,
\[|X_j|= p^{r_3-r_2+[\alpha]-j} - \lfloor p^{r_3-r_2+[\alpha]-j-1} \rfloor.\]
\end{proof}

\section{Finding $E(r_2,r_3)$}
In this section we shall find $E(r_2,r_3)$ defined in the equation preceding Lemma~\ref{lem:S}
for $(r_3,r_2)\in \mathbb{N}^2$.
We use the same terminology as in the previous section.
As mentioned before, $E(r_2,r_3)=0$ for all $(r_3,r_2)\in \mathbb{N}$ with $r_2<\min\{r_3,[\beta-\alpha]\}$.

First we deal with some cases with $[\alpha]\leq r_3$ in the following three lemmas.
In this case we have
\[\min\{[p^{r_2}-\alpha a], 2[\alpha],[\alpha]+r_3\}=\min\{[p^{r_2}-\alpha a], 2[\alpha]\}.\]

\begin{lem}\label{lem:s1}
For each $(r_3,r_2)\in \mathbb{N}^2$ with $[\beta-\alpha]+[\alpha]\leq r_2,r_3$
we have
\[E(r_2,r_3)=p^{2[\alpha]+[\beta-\alpha]}\sum_{r_1\geq 2[\alpha]}.\]
\end{lem}
\begin{proof}
If $r_3\leq r_2$, then $S(r_2,r_3)=p^{r_3-[\beta-\alpha]}\mathbb{Z}_p$ by Lemma~\ref{lem:S}.
Since $r_3\geq [\beta-\alpha]+[\alpha]$ and $[\alpha]=[\beta]\leq [\beta-\alpha]$,
it follows that, for $a\in S(r_2,r_3)$,
\[[p^{r_2}-\alpha a]\geq \min\{r_2,[\alpha]+[a]\}\geq \min\{r_2,[\alpha]+r_3-[\beta-\alpha]\}\geq 2[\alpha].\]
If $r_3>r_2$, then, by Lemma~\ref{lem:basic2} and the assumption on $r_2$, for $a\in S(r_2,r_3)$,
\[[p^{r_2}-\alpha a ]=r_2-[\beta-\alpha]+[\alpha]\geq 2[\alpha].\]
These imply that, for $a\in S(r_2,r_3)$, we have
\[\min\{[p^{r_2}-\alpha a], 2[\alpha],[\alpha]+r_3\}=2[\alpha].\]
Therefore, we conclude from Lemma~\ref{lem:cr1} and Lemma~\ref{lem:nr4} that
\begin{align*}
E(r_2,r_3)& =\sum_{\bar{a}\in \bar{S}(r_2,r_3)}\sum_{r_1\geq 2[\alpha]} p^{2[\alpha]}\\
   &=\sum_{r_1\geq 2[\alpha]} p^{2[\alpha]}|\bar{S}(r_2,r_3)|\\
   &=\sum_{r_1\geq 2[\alpha]} p^{2[\alpha]} p^{[\beta-\alpha]}.
   \end{align*}
\end{proof}

\begin{lem}\label{lem:s2}
For each $(r_3,r_2)\in \mathbb{N}^2$ with
\[[\beta-\alpha]\leq r_3,[\beta-\alpha]\leq r_2<\min\{[\beta-\alpha]+[\alpha],r_3\}\]
we have
\[E(r_2,r_3)=p^{r_2+[\alpha]}\sum_{r_1\geq r_2-[\beta-\alpha]+[\alpha]}.\]
\end{lem}
\begin{proof}
By Lemma~\ref{lem:basic2} and the assumption on $(r_3,r_2)$,
\[[p^{r_2}-\alpha a]=r_2-[\beta-\alpha]+[\alpha]\leq 2[\alpha].\]
Therefore, we conclude from Lemma~\ref{lem:cr1} and Lemma~\ref{lem:nr4} that
\begin{align*}
E(r_2,r_3)& =\sum_{\bar{a}\in \bar{S}(r_2,r_3)}\sum_{r_1\geq r_2-[\beta-\alpha]+[\alpha]} p^{r_2-[\beta-\alpha]+[\alpha]}\\
   &=\sum_{r_1\geq r_2-[\beta-\alpha]+[\alpha]} p^{r_2-[\beta-\alpha]+[\alpha]}|\bar{S}(r_2,r_3)|\\
   &=\sum_{r_1\geq r_2-[\beta-\alpha]+[\alpha]} p^{r_2-[\beta-\alpha]+[\alpha]} p^{[\beta-\alpha]}.
   \end{align*}
\end{proof}

\begin{lem}\label{lem:s3}
For each $(r_3,r_2)\in \mathbb{N}^2$ with $[\alpha]\leq r_3<[\beta-\alpha]+[\alpha], r_3+[\alpha]\leq  r_2$
we have
\[E(r_2,r_3)=\sum_{ m\leq i<[\alpha]}\sum_{r_1\geq [\alpha]+i} p^{[\alpha]+i}(p^{r_3-i}-\lfloor p^{r_3-i-1}\rfloor)+
\sum_{r_1\geq 2[\alpha]}p^{[\alpha]+r_3}\]
where $m:=\max\{0,r_3-[\beta-\alpha]\}$.
\end{lem}
\begin{proof}
Since $r_3\leq r_2$, $S(r_2,r_3)=p^m\mathbb{Z}_p$.
For $a\in S(r_2,r_3)\setminus p^{r_3}\mathbb{Z}_p$ we have
\[[\alpha]+[a]<[\alpha]+r_3\leq r_2.\]
By Lemma~\ref{lem:part}, \[[p^{r_2}-\alpha a]=[\alpha]+[a].\]
If $a\in p^{r_3}\mathbb{Z}_p$, then
 \[[p^{r_2}-\alpha a]\geq [\alpha]+[a]\geq r_3+[\alpha]\geq 2[\alpha].\]
Thus, we conclude that, for $a\in S(r_2,r_3)$,
\[\min\{[p^{r_2}-\alpha a], 2[\alpha]\}=
\begin{cases}
[\alpha]+[a] & \mbox{if $\max\{0,r_3-[\beta-\alpha]\}\leq [a]<[\alpha]$}\\
2[\alpha] & \mbox{if $[\alpha]\leq [a]$.}
\end{cases}
\]
Since $\bar{S}(r_2,r_3)$ is a disjoint union of $\bar{S}_i(r_2,r_3)$ with $m\leq i\leq r_3$,
we conclude from Lemma~\ref{lem:cr1}, Lemma~\ref{lem:nr4} and Lemma~\ref{lem:size} that
\begin{align*}
& E(r_2,r_3) \\
& =\sum_{ m\leq i\leq r_3 } \sum_{\bar{a}\in S_i(r_2,r_3)} \sum_{r_1\in C(r_2,r_3,a)} |\bar{S}(r_1,r_2,r_3,a)|\\
 & = \sum_{ m\leq i<[\alpha]} \sum_{\bar{a}\in S_i(r_2,r_3)} \sum_{r_1\geq[\alpha]+i} p^{[\alpha]+i}
 +\sum_{ [\alpha]\leq i \leq r_3} \sum_{\bar{a}\in S_i(r_2,r_3)} \sum_{r_1\geq 2[\alpha]} p^{2[\alpha]}\\
 &= \sum_{ m\leq i<[\alpha]}\sum_{r_1\geq [\alpha]+i} p^{[\alpha]+i}(p^{r_3-i}-\lfloor p^{r_3-i-1}\rfloor)
 +  \sum_{r_1\geq 2[\alpha]} p^{2[\alpha]}|\bigcup_{i=[\alpha]}^{r_3}\bar{S}_i(r_2,r_3)|\\
  &= \sum_{ m\leq i<[\alpha]}\sum_{r_1\geq [\alpha]+i}\sum_{r_1\geq [\alpha]+i} p^{[\alpha]+i}(p^{r_3-i}-\lfloor p^{r_3-i-1}\rfloor)
 + \sum_{r_1\geq 2[\alpha]} p^{2[\alpha]}p^{r_3-[\alpha]}.
\end{align*}
\end{proof}

Second, we deals with some cases with $r_3\leq [\alpha]$ in the following two lemmas.
In this case we have
\[\min\{[p^{r_2}-\alpha a], 2[\alpha],[\alpha]+r_3\}=\min\{[p^{r_2}-\alpha a], [\alpha]+r_3\}.\]
\begin{lem}\label{lem:s4}
For each $(r_3,r_2)\in \mathbb{N}^2$ with $0\leq r_3<[\alpha]$ and $[\alpha]+r_3\leq  r_2$
we have
\[E(r_2,r_3)=\sum_{ 0\leq i\leq r_3 } \sum_{r_1\geq [\alpha]+i} p^{[\alpha]+i}(p^{r_3-i}-\lfloor p^{r_3-i-1}\rfloor ).\]
\end{lem}
\begin{proof}
Since $0\leq r_3<[\alpha]$, it follows from Lemma~\ref{lem:S} that
$S(r_2,r_3)=\mathbb{Z}_p$.
For $a\in\mathbb{Z}_p$ with $[a]<r_3$ we have
\[[\alpha a]=[\alpha]+[a]<[\alpha]+r_3\leq r_2,\]
which implies that
\[
\min\{[p^{r_2}-\alpha a], [\alpha]+r_3\}=
\begin{cases}
[\alpha]+[a] & \mbox{if $[a]<r_3$}\\
[\alpha]+r_3 & \mbox{if $[a]\geq r_3$}.
\end{cases}
\]
Thus, we have
\begin{align*}
& E(r_2,r_3) \\
& =\sum_{ 0\leq i\leq r_3 } \sum_{\bar{a}\in \bar{S}_i(r_2,r_3)} \sum_{r_1\in C(r_2,r_3,a)} |\bar{S}(r_1,r_2,r_3,a)|\\
& =\sum_{ 0\leq i\leq r_3 } \sum_{\bar{a}\in \bar{S}_i(r_2,r_3)} \sum_{r_1\geq [\alpha]+i} p^{[\alpha]+i}\\
& =\sum_{ 0\leq i\leq r_3 }  \sum_{r_1\geq [\alpha]+i} p^{[\alpha]+i}(p^{r_3-i}-\lfloor p^{r_3-i-1}\rfloor ).
\end{align*}
\end{proof}

\begin{lem}\label{lem:s5}
For each $(r_3,r_2)\in \mathbb{N}^2$ with $0\leq r_3\leq  r_2<[\alpha]$
we have
\[E(r_2,r_3)=p^{r_2+r_3}\sum_{r_1\geq r_2}.\]
\end{lem}
\begin{proof}
For $a\in \mathbb{Z}_p$ we have
\[r_2<[\alpha]\leq [\alpha]+[a]=[\alpha a].\]
Thus, by Lemma~\ref{lem:part},
\[[p^{r_2}-\alpha a]=r_2<[\alpha]\leq [\alpha]+r_3.\]
Therefore,
we conclude from Lemma~\ref{lem:cr1} and Lemma~\ref{lem:nr4} that
\begin{align*}
& E(r_2,r_3) \\
& =\sum_{\bar{a}\in \bar{S}(r_2,r_3)} \sum_{r_1\in C(r_2,r_3,a)} |\bar{S}(r_1,r_2,r_3,a)|\\
& =\sum_{\bar{a}\in \bar{S}_i(r_2,r_3)} \sum_{r_1\geq r_2} p^{r_2}\\
& =\sum_{r_1\geq r_2} p^{r_2}|\bar{S}_i(r_2,r_3)|\\
& =\sum_{r_1\geq r_2} p^{r_2}p^{r_3}.\\
\end{align*}
\end{proof}
Recall that $E(r_2,r_3)$ is computed but for a finite number of $(r_3,r_2)\in \mathbb{N}^2$, i.e., but for the elements of
\begin{equation}\label{eq:map}
\{(r_3,r_2)\in \mathbb{N}^2\mid 0\leq r_3<[\beta-\alpha]+[\alpha], \max\{[\alpha],r_3\}\leq r_2<[\alpha]+r_3\}.
\end{equation}
In order to find $E(r_2,r_3)$ for $(r_3,r_2)$ in (\ref{eq:map})
we need the following formula:
According to Lemma~\ref{lem:basic} and Lemma~\ref{lem:X} we decompose $E(r_2,r_3)$
into $E_1(r_2,r_3)+E_2(r_2,r_3)+E_3(r_2,r_3)$
where $m:=\max\{0,r_3-[\beta-\alpha]\}$
\begin{align*}
  E_1(r_2,r_3) & =\sum_{m \leq i < r_2-[\alpha]} \sum_{\bar{a} \in \bar{S}_i(r_2,r_3)} \sum_{r_1 \in C(r_2,r_3,a)} |\bar{S}(r_1,r_2,r_3,a)|\\
  E_2(r_2,r_3) & =\sum_{\bar{a} \in \bar{S}_{r_2-[\alpha]}} \sum_{r_1 \in C(r_2,r_3,a)} |\bar{S}(r_1,r_2,r_3,a)|\\
  E_3(r_2,r_3) & =\sum_{r_2-[\alpha]<i \leq r_3} \sum_{\bar{a} \in \bar{S}_i(r_2,r_3)} \sum_{r_1 \in C(r_2,r_3,a)} |\bar{S}(r_1,r_2,r_3,a)|.
\end{align*}
Moreover, these operators can be computed as follows:
\begin{prop} \label{prop:e123}
For $(r_3,r_2)\in \mathbb{N}^2$ satisfying $(\ref{eq:map})$ we have the following:
\begin{align*}
&E_1(r_2,r_3)=\sum_{m \leq i < r_2-[\alpha]} \sum_{r_1 \geq l_i(r_3)} p^{l_i(r_3)} (p^{r_3-i}-\lfloor p^{r_3-i-1} \rfloor),\\
&E_2(r_2,r_3)=\\
&\sum_{0\leq j\leq r_3-r_2+[\alpha]} \sum_{r_1 \geq m_j(r_2,r_3)} p^{m_j(r_2,r_3)} (p^{r_3-r_2+[\alpha]-j}-\lfloor 1+2^{-j}\rfloor\lfloor p^{r_3-r_2+[\alpha]-j-1} \rfloor),\\
&E_3(r_2,r_3)=\sum_{r_1 \geq n(r_2,r_3)} p^{n(r_2,r_3)}p^{r_3-r_2+[\alpha]-1}
\end{align*}
where
$l_i(r_3)= \textup{min}\{[\alpha]+i,2[\alpha],r_3+[\alpha]\}$,\\
$m_j(r_2,r_3)  =\textup{min}\{r_2+j,2[\alpha],r_3+[\alpha]\}$ and
$n(r_2,r_3)=\textup{min}\{r_2,2[\alpha],r_3+[\alpha]\}$.\end{prop}
\begin{proof}
For the cases $E_1(r_2,r_3)$, $E_2(r_2,r_3)$ and $E_3(r_2,r_3)$,
$|\bar{S}(r_1,r_2,r_3,a)|$ does not depend on the choice of
\[\mbox{$\bar{a}\in \bar{S}_i(r_2,r_3)$ for $i=m,m+1,\ldots, r_2-[\alpha]-1$,}\]
\[\mbox{
$\bar{a}\in X_j$ for $j=0,1,\ldots, r_3-r_2 +[\alpha]$,}\]
\[\mbox{ and
$\bar{a}\in \bar{S}_i(r_2,r_3)$ for $i=r_2-[\alpha]+1, r_2-[\alpha]+2,\ldots, r_3$,
respectively.}\] This implies that, for $i=1,2,3$,
the equation $E_i(r_2,r_3)$ is obtained as an immediate consequence of
Lemma~\ref{lem:cr1}, Lemma~\ref{lem:nr4}, Lemma~\ref{lem:basic} and Lema~\ref{lem:X}.
\end{proof}

Proposition~\ref{prop:e123} shows that $\zeta_R(s)$ can be theoretically computed as follows:
\begin{thm}\label{thm:main}
We have
\begin{equation*}%\label{eq:last}
\zeta_R(s)=\sum_{i=1}^6\sum_{(r_3,r_2)\in R_i}E(r_2,r_3)p^{-s(r_1+r_2+r_3)}
\end{equation*}
where
\begin{align*}
R_1&=\{(r_3,r_2)\mid[\beta-\alpha]+[\alpha]\leq r_2,r_3\}\\
R_2&=\{(r_3,r_2)\mid [\beta-\alpha]\leq r_3,[\beta-\alpha]\leq r_2<\min\{[\beta-\alpha]+[\alpha],r_3\}\}\\
R_3&=\{(r_3,r_2)\mid [\alpha]\leq r_3<[\beta-\alpha]+[\alpha], r_3+[\alpha]\leq  r_2\}\\
R_4&=\{(r_3,r_2)\mid 0\leq r_3<[\alpha], [\alpha]+r_3\leq  r_2\}\}\\
R_5&=\{(r_3,r_2)\mid 0\leq r_3\leq  r_2<[\alpha]\}\}\\
R_6&=\{(r_3,r_2)\mid 0\leq r_3<[\beta-\alpha]+[\alpha], \max\{[\alpha],r_3\}\leq  r_2<[\alpha]+r_3\}\}.
\end{align*}
\end{thm}
Combining calculations on power series
we obtain an explicit form of $\zeta_R(s)$.
For example, letting $x=p^{-s}$,
\begin{align*}
& \sum_{(r_3,r_2)\in R_1}E(r_2,r_3)p^{-s(r_1+r_2+r_3)}\\
& =\sum_{r_3\geq [\beta-\alpha]+[\alpha]}\sum_{r_2\geq [\beta-\alpha]+[\alpha]}p^{2[\alpha]+[\beta-\alpha]}\sum_{r_1\geq 2[\alpha]}p^{-s(r_1+r_2+r_3)}\\
& =p^{2[\alpha]+[\beta-\alpha]}\sum_{r_3\geq [\beta-\alpha]+[\alpha]}p^{-sr_3}
\sum_{r_2\geq [\beta-\alpha]+[\alpha]}p^{-sr_2}\sum_{r_1\geq 2[\alpha]}p^{-sr_1}\\
& =p^{2[\alpha]+[\beta-\alpha]}
\frac{x^{[\beta-\alpha]+[\alpha]}}{1-x}
\frac{x^{[\beta-\alpha]+[\alpha]}}{1-x}
\frac{x^{2[\alpha]}}{1-x}\\
& =\frac{(px^2)^{2[\alpha]+[\beta-\alpha]}}{(1-x)^3}.
\end{align*}
\section{Some applications}
We use the same notation on $R$ as in the previous section
and aim to find an explicit form of $\zeta_R(s)$ under certain assumption on
$([\alpha],[\beta],[\beta-\alpha])$.
In this section we always denote $p^{-s}$ by $x$.
\begin{prop}\label{prop:001}
If $([\alpha],[\beta],[\beta-\alpha])=(0,0,k)$,
then
\[\zeta_R(s)=\frac{1-p^kx^{2k}}{1-px^2}\frac{1}{(1-x)^2}+\frac{p^kx^{2k}}{(1-x)^3}.\]
\end{prop}
\begin{proof}
Since $[\alpha]=[\beta]=0$,
we have $R_i=\emptyset$ for $i=2,4,5,6$.
We have already computed $\sum_{(r_3,r_2)\in R_1}E(r_2,r_3)p^{-s(r_1+r_2+r_3)}$ in Section~3, and
we obtain from Lemma~\ref{lem:s2} that
\begin{align*}
& \sum_{(r_3,r_2)\in R_3}E(r_2,r_3)p^{-s(r_1+r_2+r_3)}\\
& =\sum_{[\alpha]\leq r_3<[\beta-\alpha]+[\alpha]}\sum_{r_2\geq r_3+[\alpha]}\sum_{r_1\geq 2[\alpha]}p^{[\alpha]+r_3}p^{-s(r_1+r_2+r_3)}\\
& =\sum_{0\leq r_3<k}\sum_{r_2\geq r_3}\sum_{r_1\geq 0}p^{r_3}p^{-s(r_1+r_2+r_3)}\\
& =\sum_{0\leq r_3<k}p^{r_3(1-s)}\sum_{r_2\geq r_3}p^{-sr_2}\sum_{r_1\geq 0}p^{-sr_1}\\
& =\sum_{0\leq r_3<k}p^{r_3(1-s)}  \frac{p^{-sr_3}}{1-x}\frac{1}{1-x}\\
& =\sum_{0\leq r_3<k}p^{r_3(1-2s)}  \frac{1}{1-x}\frac{1}{1-x}\\
& =\frac{1-p^k x^{2k}}{1-px^2}\frac{1}{(1-x)^2}.
\end{align*}
Thus, $\zeta_R(s)=\sum_{i=1,3}\sum_{(r_3,r_2)\in R_i}E(r_2,r_3)p^{-s(r_1+r_2+r_3)}$ as desired.
\end{proof}

As an application of Proposition~\ref{prop:001} with $k=1$ we
show a way to find the zeta function given in Example~\ref{ex:peter}.
Let $B$ denote the adjacency matrix of the peterson graph.
Since $B$ has the eigenvalues $3$, $1$ and $-2$,
$\zeta_{\gn{B}}(s)$ equals the zeta function of
$\mathbb{Z}[x]/x(x+3)(x-2)\mathbb{Z}[x]$.
Note that, for each prime divisor $p$ of $(-3)2\cdot 5$,
we have $([\alpha],[\beta-\alpha])=(0,1)$ and
by Lemma~\ref{lem:B} and Proposition~\ref{prop:001}
\[\delta_p(x)=\frac{(1-x+px^2)}{(1-x)^3}\cdot \frac{1}{(1-x)^{-3}}=1-x+px^2.\]
Therefore, we have
\[\zeta_{\gn{B}}(s)=\prod_{p=2,3,5}(1-p^{-s}+p^{1-2s})\cdot \zeta(s)\cdot \zeta(s)\cdot \zeta(s).\]

\begin{prop}\label{prop:002}
If $([\alpha],[\beta],[\beta-\alpha])=(1,1,1)$,
then
\begin{align*}
&(1-x)^3\zeta_R(s)=\\
&p^3x^6-2p^2x^5+(p^2+p)x^4+(p^2-2p)x^3+(p+1)x^2-2x+1.
\end{align*}
\end{prop}
\begin{proof}
Applying the first five lemmas in Section~3 we obtain
\[\eta_i(p^{-s}):=\sum_{(r_3,r_2)\in R_i}E(r_2,r_3)p^{-s(r_1+r_2+r_3)}\]
 with $1\leq i\leq 6$ as follows:
\begin{align*}
\eta_1(x) & =\frac{p^3x^6}{(1-x)^3} &
\eta_2(x) & =\frac{p^2x^4}{(1-x)^2}\\
\eta_3(x) & =\frac{(p^2-p)x^4+p^2x^5 }{(1-x)^2}\\
\eta_4(x) & =\frac{px^2}{(1-x)^2} &
\eta_5(x) & =\frac{1}{1-x}
\end{align*}
Since $R_6=\{(1,1)\}$, we apply Proposition~\ref{prop:e123} for $E_k(1,1)$ with $k=1,2,3$
to obtain the following:
\begin{align*}
E_1(1,1)=0 & \hspace{3mm} & E_2(1,1)=p(p-2)\sum_{r_1\geq 1}+p^2\sum_{r_1\geq 2} & \hspace{3mm} &  E_3(1,1)= p\sum_{r_1\geq 1}\\
\end{align*}
\[\eta_6(x)=\frac{(p^2-2p)x^3+p^2x^4+px^3}{1-x}.\]
Thus,
\begin{align*}
\zeta_R(s)&=\sum_{(r_3,r_2)\in \mathbb{N}^2 } E(r_2,r_3) p^{-s(r_1+r_2+r_3)}\\
           &  =\sum_{i=1}^6\eta_i(x)\\
         & =\frac{p^3x^6+(1-x)(p^2x^4+(p^2-p)x^4+p^2x^5+px^2)}{(1-x)^3}\\
         & +\frac{(1-x)^2(1+(p^2-2p)x^3+p^2x^4+px^3)}{(1-x)^3}\\
         &=\frac{p^3x^6+(1-x)(px^2+(2p^2-p)x^4+p^2x^5)}{(1-x)^3}\\
         & +\frac{(1-x)^2(1+(p^2-p)x^3+p^2x^4)}{(1-x)^3}\\
         &=\frac{p^3x^6-2p^2x^5+(p^2+p)x^4+(p^2-2p)x^3+(p+1)x^2-2x+1}{(1-x)^3}
\end{align*}
which equals the formula given in the statement.
\end{proof}

\begin{prop}\label{prop:003}
If $([\alpha],[\beta],[\beta-\alpha])=(1,1,2)$,
then
\begin{align*}
(1-x)^3\zeta_R(s)=
& p^4x^8-2p^3x^7+(p^3+p^2)x^6+(p^3-p^2)x^5\\
& +px^4+(p^2-2p)x^3+(p+1)x^2-2x+1.
\end{align*}
\end{prop}
\begin{proof}
Applying the first five lemmas in Section~3 we obtain
\[\eta_i(p^{-s}):=\sum_{(r_3,r_2)\in R_i}E(r_2,r_3)p^{-s(r_1+r_2+r_3)}\]
 with $1\leq i\leq 6$ as follows:
\begin{align*}
\eta_1(x) & =\frac{p^4x^8}{(1-x)^3} &
\eta_2(x) & =\frac{p^3x^6}{(1-x)^2}\\
\eta_3(x) & =\frac{p^3x^7+(p^3-p^2)x^6+p^2x^5+(p^2-p)x^4 }{(1-x)^2}\\
\eta_4(x) & =\frac{px^2}{(1-x)^2} &
\eta_5(x) & =\frac{1}{1-x}
\end{align*}
Since $R_6=\{(1,1),(2,2)\}$, we apply Proposition~\ref{prop:e123} for $E_k(i,i)$ with $i=1,2$ and $k=1,2,3$
to obtain the following:
\begin{align*}
E_1(1,1)=0 & \hspace{3mm} & E_2(1,1)=p(p-2)\sum_{r_1\geq 1}+p^2\sum_{r_1\geq 2}\\
E_3(1,1)=p\sum_{r_1\geq 1} &\hspace{3mm} & E_1(2,2)=(p^3-p^2)\sum_{r_1\geq 1}\\
E_2(2,2)=(p^3-p^2)\sum_{r_2\geq 2} & \hspace{3mm}  & E_3(2,2)=\sum_{r_1\geq 2}p^2.
\end{align*}
Thus,
\begin{align*}
\eta_6(x) & =\frac{0+p(p-2)x^3+p^2x^4+px^3+(p^3-p^2)x^5+(p^3-p^2)x^6+p^2x^6}{1-x}\\
          & =\frac{(p^2-p)x^3+p^2x^4+(p^3-p^2)x^5+p^3x^6}{1-x}.
\end{align*}
Therefore,
\begin{align*}
& \zeta_R(s)\\
 & =\sum_{(r_3,r_2)\in \mathbb{N}^2 } E(r_2,r_3) p^{-s(r_1+r_2+r_3)}\\
           &  =\sum_{i=1}^6\eta_i(x)\\
          &  =\frac{p^4 x^8+p^3 x^6(1-x)}{(1-x)^3}\\
          &  +\frac{(1-x)(p^3x^7+(p^3-p^2)x^6+p^2x^5+(p^2-p)x^4)}{(1-x)^3}\\
          &  +\frac{px^2(1-x)+(1-x)^2+
           (1-x)^2((p^2-p)x^3+p^2x^4+(p^3-p^2)x^5+p^3x^6)}{(1-x)^3},
\end{align*}
which equals the formula given in the statement.
\end{proof}
As a direct consequence of Proposition~\ref{prop:003} with $p=2$
we obtain the same formula as in Example~\ref{ex:quad}
since the adjacency matrix of the quadrangle has its eigenvalues $0,2,-2$.
One may notice that the last three propositions enables us to find
an explicit from of the zeta function $\zeta_{\gn{B}}(s)$ whenever $B$ has exactly three eigenvalues $k$,$r$ and $s$ from integers and,
for each prime divisor $p$ of $(r-k)(s-k)(r-s)$, $([r-k],[s-k],[r-s])$ can be permutated as one of the forms given in the three propositions of this section. For example,
$B$ is the adjacency matrix of one of the following graphs:
$K_{3,3}$, $K_{2,2,2}$, $K_{5,5}$, $Payley(9)$,
$T(5)$,$K_{6,6}$, $K_{4,4,4}$, $K_{3,3,3,3}$, $K_{2,2,2,2,2,2}$, $Sp(4,2)$, $T(6)$, and so on
(see \cite{c} to find the eigenvalues of strongly-regular graphs with small number of vertices
and the terminology to represent graphs).

%%%%%%%%%%%%%%%%%%%%%%%%%%%%%%%%%%%%%%%%%%%%%%%%%%%%%%%


\begin{thebibliography}{99}

\bibitem{bi} E.~Bannai and T.~Ito,
\textit{Algebraic Combinatorics. I. Association Schemes},
The Benjamin/Cummings Publishing Co., Inc., Menlo Park, CA, 1984.

\bibitem{c}
\textit{Handbook of combinatorial designs.
Edited by Charles J.~Colbourn and J.H.~Dinitz, Second edition. Discrete Mathematics and its Applications (Boca Raton)},
Chapman \& Hall/CRC, Boca Raton, FL, 2007.
\bibitem{Dam}
E.~van Dam, Nonregular graphs with three eigenvalues,
\textit{J. Combin. Theory Ser. B} 73 (1998), no. 2, 101--118.
%\bibitem{kissme}
%A.Hanaki, I.Miyamoto, \textit{Classification of association schemes with small verticies},
%{\tt http://kissme.shinshu-u.ac.jp/as/}.



\bibitem{hh}
A.~Hanaki, M.~Hirasaka
Zeta functions of adjacency algebras of association schemes
of prime order or rank two, Hokkaido Mathematical Journal, accepted in 2014.

\bibitem{hironaka}
Y.~Hironaka,
Zeta functions of integral group rings of metacyclic groups,
\textit{Tsukuba J. Math.} 5 (1981), no. 2, 267--283.

%\bibitem{komatsu}
%T.~Komatsu,
%Tamely ramified Eisenstein fields with prime power discriminants. (English summary)
%Kyushu J. Math, 62 (2008), no. 1, 1-13.
\bibitem{KM} M.~Klin, M.~Muzychuk, On graphs with three eigenvalues,
\textit{Discrete Mathematics} 189 (1998) 191-207.


\bibitem{oh}
S.~Oh,  The number of ideals of $\mathbb{Z}[x]/(x^3-4x)$,
Master thesis, 2013.

\bibitem{reiner}
I.~Reiner,
Zeta functions of integral representations,
\textit{Comm. Algebra} 8 (1980), no. 10, 911--925.

\bibitem{solomon} L.~Solomon,
Zeta functions and integral representation theory,
\textit{Advances in Math.} 26 (1977), no. 3, 306--326.

\bibitem{take}
Y.~Takegahara,
Zeta functions of integral group rings of abelian (p,p)-groups.
\textit{Comm. Algebra} 15 (1987), no. 12, 2565--2615.


\bibitem{Zie1}
P.-H.~Zieschang, \textit{An algebraic approach to association schemes},
Lecture Notes in Mathematics, vol. 1628, Springer-Verlag, Berlin, 1996.

\bibitem{Zie2}
P.-H.~Zieschang, \textit{Theory of association schemes}, Springer
Monograph in Mathematics, Springer-Verlag, Berlin, 2005.



\end{thebibliography}
\end{document}